\documentclass[11pt]{article}

\usepackage{amsmath, amssymb, amsthm} 

\allowdisplaybreaks

\theoremstyle{plain}
\newtheorem{theorem}{Theorem}[section]
\newtheorem{lemma}[theorem]{Lemma}

\newtheorem{corollary}[theorem]{Corollary}

\newcommand{\poly}{\text{poly}}
\RequirePackage[normalem]{ulem} 
\RequirePackage{color}\definecolor{RED}{rgb}{1,0,0}\definecolor{BLUE}{rgb}{0,0,1} 

\title{Modifying a Graph's Degree Sequence and the Testablity of Degree Sequence Properties}

\author{Lior Gishboliner\thanks{ETH Zurich. 
		Email: lior.gishboliner@math.ethz.ch.}
	}

\begin{document}
	\maketitle
	
	\begin{abstract}
		We show that if the degree sequence of a graph $G$ is close in $\ell_1$-distance to a given realizable degree sequence $(d_1,\dots,d_n)$, then $G$ is close in edit distance to a graph with degree sequence $(d_1,\dots,d_n)$. We then use this result to prove that every graph property defined in terms of the degree sequence is testable in the dense graph model with query complexity independent of $n$.  
	\end{abstract}
	
	\section{Introduction}\label{sec:intro}
	Our main result in this paper is concerned with efficiently modifying a graph (i.e. adding/deleting edges) so as to obtain a graph with a prescribed degree sequence. Let us begin by introducing the relevant definitions. 
	For convenience, we let the vertex-set of $n$-vertex graphs be $[n]$. The degree sequence of an $n$-vertex graph $G$ is $(d_G(1),\dots,d_G(n))$, where $d_G(i)$ is the degree of $i$ in $G$. A sequence $(d_1,\dots,d_n)$ is called {\em realizable} or {\em graphic} if there is an $n$-vertex graph $G$ such that $d_G(i) = d_i$ for every $1 \leq i \leq n$. Graphic sequences are a classical object of study, and there are several famous theorems characterizing them, such as the Erd\H{o}s-Gallai theorem \cite{EG} and the Havel-Hakimi theorem \cite{Hakimi,Havel}. See \cite{HS} for more information. 
	The normalized $\ell_1$-distance between sequences $x = (x_1,\dots,x_n)$ and $y = (y_1,\dots,y_n)$ is defined as $\ell_1(x,y) := \frac{1}{n^2} \cdot \sum_{i = 1}^{n}|x_i - y_i|$. Let $G,G'$ be graphs on $[n]$. The (normalized) edit distance between $G$ and $G'$ is defined as $|E(G) \triangle E(G')|/n^2$. 
	
	We consider the following natural question: given that the degree sequence of a graph $G$ is close in $\ell_1$-distance to some realizable sequence $(d_1,\dots,d_n)$, how close is $G$ (in edit distance) to a graph whose degree sequence is $(d_1,\dots,d_n)$? As far as we know, this question has so far only been considered for the constant sequence $(r,\dots,r)$, which evidently corresponds to an $r$-regular graph (see \cite[Claim 8.5.1]{Goldreich}). Here we obtain a general result for all degree sequences:  
	
	\begin{theorem}\label{thm:distance}
		Let $\delta > 0$ and $n \geq \delta^{-2}$, let $G$ be a graph on $[n]$, let $(d_1,\dots,d_n)$ be a realizable sequence, 
		and suppose that $\sum_{i \in V(G)}{|d_G(i) - d_i|} \leq \delta n^2$.
		Then there is a graph $G'$ on $V(G) = [n]$ such that $|E(G') \triangle E(G)| = O(\delta^{1/2}n^2)$ and $d_{G'}(i) = d_i$ for every $i \in [n]$. 
	\end{theorem}
	Note that Theorem \ref{thm:distance} gives a quadratic dependence; if the $\ell_1$-distance between the degree sequences is $\delta$, then the edit distance between $G$ and $G'$ is $O(\delta^{1/2})$. In the special case where $(d_1,\dots,d_n)$ is a constant sequence (namely, where we consider the distance to being $r$-regular for some $r$), a better, linear dependence has been proved (see Claim 8.5.1 in \cite{Goldreich}). 
	We wonder if one can obtain a linear dependence for all degree sequences, thus proving Theorem \ref{thm:distance} with an optimal bound.
	
	
	\subsection{Every degree-sequence property is testable}
	
	
	
	Theorem \ref{thm:distance} has an immediate application to {\em graph property testing}. This is an area of theoretical computer science concerned with the design of fast, randomized algorithms which distinguish graphs satisfying a certain property from graphs which are far from the property. In the so-called {\em dense graph model}, which is the setting considered here, the measure of distance is simply the (normalized) edit distance defined above. 
	More precisely, we say that a graph $G$ is {\em $\varepsilon$-close} to a graph property $\mathcal{P}$ if there is a graph satisfying $\mathcal{P}$ whose edit distance to $G$ is at most $\varepsilon$; otherwise, $G$ is {\em $\varepsilon$-far} to $\mathcal{P}$. 
	A {\em tester} for $\mathcal{P}$ is an algorithm which, given an input consisting of a graph $G$ and a proximity parameter $\varepsilon > 0$, accepts with probability at least $\frac{2}{3}$ if $G$ satisfies $\mathcal{P}$, and rejects with probability at least $\frac{2}{3}$ if $G$ is $\varepsilon$-far from $\mathcal{P}$. The tester accesses the graph $G$ by making edge-queries to its adjacency matrix, and may also sample random vertices of $G$. As is customary, we assume that $n$, the number of vertices of the input graph, is given to the tester as part of the input. 
	We say that a tester has {\em query complexity} $q(n,\varepsilon)$ if it makes at most $q(n,\varepsilon)$ edge-queries whenever invoked with an $n$-vertex input graph and with proximity parameter $\varepsilon$. In the present paper, we focus on testers whose query complexity is independent of $n$, namely, can be bounded by a function $q(\varepsilon)$ of $\varepsilon$ alone. In such a case, we will say that the query complexity of the tester is $q(\varepsilon)$. 
	The study of graph property testers was initiated in the seminal paper of Goldreich, Goldwasser and Ron \cite{GGR}. 
	We must mention that property testing is a much wider area of research than is presented here, and that
	apart from the dense graph model considered in this paper, property testers have been extensively studied in numerous other settings. For an overview of the field, we refer the reader to the book of Goldreich \cite{Goldreich}. 
	
	
	In the present paper we are concerned with properties defined in terms of the degree sequence. 
	The precise definitions are as follows. 
	A {\em degree-sequence property} is a set $\mathcal{D}$ of realizable sequences that is closed under permuting the coordinates, which means that for every $(d_1,\dots,d_n) \in \mathcal{D}$ and every permutation $\pi : [n] \rightarrow [n]$, the permuted sequence $(d_{\pi(1)},\dots,d_{\pi(n)})$ is also in $\mathcal{D}$. The graph property defined by $\mathcal{D}$, which we denote by $\mathcal{P}(\mathcal{D})$, is the set of all graphs whose degree sequence is in $\mathcal{D}$. 
	Our main result establishes that {\em every} degree-sequence property is testable with query complexity which is independent of $n$, and moreover depends only polynomially on $\varepsilon$. 
	\begin{theorem}\label{thm:main}
		Every degree property is testable with query complexity $q(\varepsilon) = \poly(1/\varepsilon)$. 
	\end{theorem}

	A significant caveat to Theorem \ref{thm:main} is that while the {\em query complexity} of the tester supplied by this theorem is independent of $n$, its {\em time complexity} may depend on $n$, and in some cases quite badly; it may be as large as exponential in $n$. The reason is that our tester for $\mathcal{P}(\mathcal{D})$ works by first constructing a bounded-size approximation of the degree sequence (of size depending only on $\varepsilon$), and then checking whether this approximation is close (in $\ell_1$-distance) to one of the sequences in $\mathcal{D}$. This second step is the source of (possibly) large time complexity: for general degree-sequence properties, we are not aware of a way of performing this check more efficietly than simply going over all sequences in $\mathcal{D}$. The required run-time may thus be as large as the number of $n$-term realizable sequences, which is at least exponential in $n$ 
	(see, for example, \cite{Burns}).
	Still, for some specific degree-sequence properties $\mathcal{D}$, this task can be done much faster, in time which depends only on $\varepsilon$. 
	One example is the case that $\mathcal{P}(\mathcal{D})$ is the property of being $r$-regular (for some $r = r(n)$); see the explanation at the end of Section \ref{sec:test}.
	We should mention that, as demonstrated in \cite[Section 8.2.3]{Goldreich}, $r$-regularity was already known to be testable with query- and time-complexity depending only on $\varepsilon$ (and in fact with a better dependence than is given by our general Theorem \ref{thm:main}).
	It may be interesting to find general families of degree-sequence properties for which the run-time is independent of $n$. On the other hand, it seems likely that dependence on $n$ is unavoidable if one wishes to handle all degree-sequence properties. It may be interesting to determine if there are cases where exponential dependence is necessary. 
	
	\paragraph{Paper organization} Theorem \ref{thm:distance} is proved in Section \ref{sec:distance}, and Theorem \ref{thm:main} is proved in Section \ref{sec:test}, where it is restated as Theorem \ref{thm:testing}. 

	\section{Modifying the Degree Sequence}\label{sec:distance}
	
	
	
	In the proof of Theorem \ref{thm:distance} we will consider 2-edge-coloured graphs, namely graphs whose edges are coloured by red and blue. For a vertex $v$ in a $2$-edge-coloured graph $G$, we denote by $N^R_G(v)$ (resp. $N^B_G(v)$) the red (resp. blue) neighbourhood of $v$ in $G$; namely, $$N^R_G(v) = \{u \in V(G) : \{u,v\} \in E(G) \text{ is coloured red}\},$$ and similarly for blue. We will also put $d^R_G(v) = |N^R_G(v)|$ and $d^B_G(v) = |N^B_G(v)|$. 
	When there is no superscript indicating the colour, we mean that both colours are accounted; so $N_G(v) = N_G^R(v) \cup N_G^B(v)$ and 
	$d_G(v) = d_G^R(v) + d_G^B(v)$.  
	The subscript $G$ will be omitted when there is no risk of confusion.  
	For a set $U \subseteq V(G)$, we will denote by $G[U]$ the (2-edge-coloured) subgraph of $G$ induced by $U$. We will use $e(U)$ to denote the number of edges inside a vertex-set $U$, and $e(U,W)$ to denote the number of edges between vertex-sets $U,W$.
	All logarithms are base $2$.
	
	An {\em alternating cycle} in a 2-edge-coloured graph is a cycle whose edge have alternating colours (in particular, such a cycle must be even). 
	We will use the following result of Grossman and H\"{a}ggkvist \nolinebreak \cite{GH}. 
	\begin{theorem}[\cite{GH}]\label{thm:alternating_cycle}
			Let $F$ be a $2$-connected $2$-edge-coloured graph in which $d^R(v),d^B(v) \geq 1$ for every $v \in V(F)$. Then $F$ contains an alternating cycle. 
	\end{theorem}
	\noindent
	The following is an easy corollary of Theorem \ref{thm:alternating_cycle}. 
	\begin{corollary}\label{cor:alternating_cycle}
			Let $F$ be a $2$-edge-coloured graph on $n \geq 2$ vertices in which $d^R(v),d^B(v) \geq \log n$ for every $v \in V(F)$. Then $F$ contains an alternating cycle. 
	\end{corollary}
	\begin{proof}
	The proof is by induction on $n$, where the base case $n=2$ is trivial. Suppose then that $n \geq 3$.  
	By Theorem \ref{thm:alternating_cycle}, if $F$ has no alternating cycle then it is not $2$-connected, 
	meaning that there is a partition $V(F) = X \cup Y \cup S$ such that $|S| \leq 1$, $X,Y \neq \emptyset$, and there are no edges in $F$ between $X$ and $Y$. Suppose without loss of generality that $|X| \leq n/2$ (the case $|Y| \leq n/2$ is symmetric). As there are no edges between $X$ and $Y$ and as $|S| \leq 1$, we must have 
	$d^R_{F[X]}(v),d^B_{F[X]}(v) \geq \log n - 1 = \log(n/2) \geq \log(|X|)$ for every $v \in X$. In particular we have $|X| \geq 2$, since $\log n > 1$. By the induction hypothesis, $F[X]$ contains an alternating cycle, as required.      
	\end{proof}

	An alternative way of deriving Corollary \ref{cor:alternating_cycle} is to apply a directed analogue thereof due to Gutin, Sudakov and Yeo \cite{GSY}, which states that every $n$-vertex $2$-edge-coloured digraph with minimum out- and in-degree at least $C\log n$ in each of the colours contains an alternating directed cycle. The reduction from the undirected case to the directed one proceeds by replacing each undirected edge $\{u,v\}$ with both directed edges $(u,v),(v,u)$, coloured with the same colour as $\{u,v\}$.   
	

	
	The last tool we need in the proof of Theorem \ref{thm:distance} is the following lemma.

	\begin{lemma}\label{lem:balanced}
		Let $\delta > 0$ and $n \geq \delta^{-2}$, let $F$ be a $2$-edge-coloured graph on $[n]$ which contains no alternating cycle, and suppose that
		$\sum_{i \in V(F)}{|d^R(i) - d^B(i)|} \leq \delta n^2$. Then $|E(F)| \leq O(\delta^{1/2}n^2)$. 
	\end{lemma}
	\begin{proof}
		By (the contrapositive of) Corollary \ref{cor:alternating_cycle}, there exists {\em no} non-empty subset $U \subseteq V(F)$ such that $d^R_{F[U]}(v),d^B_{F[U]}(v) \geq \log n \geq \log(|U|)$ 
		for every $v \in U$. This implies that there is an ordering of the vertices of $F$ such that each vertex has either less than $\log n$ red edges or less than $\log n$ blue edges to the vertices succeeding it in the ordering. By possibly renaming vertices, we may assume that this ordering is $1,\dots,n$; namely, for every $1 \leq i \leq n$ there is a colour $c_i \in \{ \text{red,blue} \}$, such that $i$ has less than $\log n$ neighbours in colour $c_i$ in the set $\{i+1,\dots,n\}$. For an edge $e = \{i,j\} \in E(F)$ with $i < j$, we will call $i$ the {\em left end} of $e$ and $j$ the {\em right end} of $e$. 
		Let $E^*$ be the set of edges of the form $\{i,j\}$, $i < j$, whose colour is $c_i$. Note that $|E^*| \leq n\log n$.  
		
		Set $k := \delta^{-1/2}$, and partition $[n]$ into $k$ consecutive intervals $X_1,\dots,X_k$ of length $\frac{n}{k} = \delta^{1/2}n$ each. Namely, for each $1 \leq j \leq k$, define 
		$X_j := \{ i \in [n] : (j-1) \cdot \frac{n}{k} + 1 \leq i \leq j \cdot \frac{n}{k}\}$. 
		For each $1 \leq j \leq k$, put 
		$
		\Delta_j := \sum_{i \in X_j}{|d^R(i) - d^B(i)|}$, noting that $\Delta_1 + \dots + \Delta_k \leq \delta n^2$ by assumption. Finally, define $R_j := \{i \in X_j : c_i = \text{red}\}$ and 
		$B_j := \{i \in X_j : c_i = \text{blue}\}$ (for $1 \leq j \leq k$).
		
		Fixing any $1 \leq j \leq k$, we claim that
		\begin{equation}\label{eq:forward_edges}
		\sum_{i \in R_j}{d^B(i)} + \sum_{i \in B_j}{d^R(i)} \geq 
		\sum_{t = j+1}^{k}{e(X_j,X_{t})} + e(X_j) - n\log n.
		\end{equation}
		To see that \eqref{eq:forward_edges} holds, recall that each $i \in R_j$ is the left end of at most $\log n$ red edges, and each $i \in B_j$ is the left end of at most $\log n$ blue edges. Thus, apart from the $|E^*| \leq n\log n$ edges in $E^*$, all edges whose left end is in $X_j$  --- i.e., all edges inside $X_j$ or between $X_j$ and $X_{j+1} \cup \dots \cup X_{k}$ --- are either blue with a left end in $R_j$ or red with a left end in $B_j$, and are hence counted by the left-hand side of \eqref{eq:forward_edges}. 
		
		Next, we use a similar consideration to argue that
		\begin{equation}\label{eq:backward_edges} 
		\sum_{i \in R_j}{d^R(i)} + \sum_{i \in B_j}{d^B(i)} \leq 
		\sum_{s = 1}^{j-1}{e(X_{s},X_j)} + e(X_j) + 2n\log n.
		\end{equation}
		To see that \eqref{eq:backward_edges} holds, 
		we again use the fact that each $i \in R_j$ is the left end of at most $\log n$ red edges and each $i \in B_j$ is the left end of at most $\log n$ blue edges, which means that apart from the $|E^*| \leq n \log n$ edges in $E^*$ (which are counted at most twice), every edge counted by the left-hand side of \eqref{eq:backward_edges} has its right end in $X_j$, and hence its left end in $X_1 \cup \dots \cup X_j$. Furthermore, if $e = \{x,y\}$ is contained in $X_j$ and $e \notin E^*$, then $e$ is counted at most once by the left-hand side of \eqref{eq:backward_edges}; indeed, assuming $x < y$, we know that either $x \in R_j$ and $e$ is blue, or $x \in B_j$ and $e$ is red, which means that the sums on the left-hand side of \eqref{eq:backward_edges} may only count $e$ when $i = y$, but not when $i = x$. This establishes \eqref{eq:backward_edges}. 
		
		Finally, our definition of $\Delta_j$ implies that 
		\begin{equation}\label{eq:discrepancy}
		\sum_{i \in R_j}{d^B(i)} + \sum_{i \in B_j}{d^R(i)} \leq 
		\sum_{i \in R_j}{d^R(i)} + \sum_{i \in B_j}{d^B(i)} + \Delta_j.
		\end{equation}
		By combining \eqref{eq:forward_edges}, \eqref{eq:backward_edges} and \eqref{eq:discrepancy}, we obtain
		\begin{equation}\label{eq:main}
		\sum_{t = j+1}^{k}{e(X_j,X_{t})} \leq 
		\sum_{s = 1}^{j-1}{e(X_{s},X_j)} + \Delta_j + 3n\log n.
		\end{equation}
		Now fix any $1 \leq r \leq k-1$, and sum the inequality \eqref{eq:main} over $j = 1,\dots,r$ to obtain:
		$$
		\sum_{j = 1}^{r}\sum_{t = j+1}^{k} e(X_j,X_{t}) \leq 
		\sum_{1 \leq s < j \leq r}e(X_{s},X_j) + \sum_{j = 1}^{r}{\Delta_j} + 3rn\log n.
		$$
		Observe that for every $1 \leq s < j \leq r$, the term $e(X_{s},X_j)$ appears in the left-hand side above, and thus may be canceled out. On the other hand, the left-hand side has the terms $e(X_j,X_{r+1})$, $1 \leq j \leq r$, which do not appear on the right-hand side. Note also that $\Delta_1 + \dots + \Delta_r \leq \Delta_1 + \dots + \Delta_k \leq \delta n^2$ and $3rn\log n \leq 3kn\log n \leq O(\delta n^2)$, where the last inequality holds due to our assumption that $n \geq \delta^{-2}$ and as $k = \delta^{-1/2}$. Combining all the above, after the aforementioned cancellations we obtain
		$
		\sum_{j = 1}^{r}{e(X_j,X_{r+1})} = O(\delta n^2).
		$
		Summing this over all $1 \leq r \leq k-1$, we get
		$$
		\sum_{1 \leq s < t \leq k}{e(X_s,X_t)} = O(k \delta n^2) = O(\delta^{1/2} n^2).
		$$
		Finally, note that $\sum_{j = 1}^{k}{e(X_j)} \leq k \cdot (n/k)^2 = n^2/k = \delta^{1/2}n^2$. Altogether, we have
		$$
		e(F) = \sum_{j = 1}^{k}{e(X_j)} + \sum_{1 \leq s < t \leq k}{e(X_s,X_t)} = O(\delta^{1/2}n^2),
		$$
		as required.  
	\end{proof}

	\noindent
	We are now ready to prove Theorem \ref{thm:distance}.

	\begin{proof}[Proof of Theorem \ref{thm:distance}]
		Since $(d_1,\dots,d_n)$ is realizable, there is a graph on $[n]$ in which the degree of vertex $i$ is $d_i$ (for every $1 \leq i \leq n$). 
		Among all such graphs, fix one, $G'$, which minimizes 
		$|E(G') \triangle E(G)|$. Our goal is to show that 
		$|E(G') \triangle E(G)| = O(\delta^{1/2}n^2)$. 
		Let $F$ be the graph on $[n]$ with edge-set $E(F) = E(G') \triangle E(G)$. We $2$-colour the edges of $F$ by colouring $E(G') \setminus E(G)$ red and $E(G) \setminus E(G')$ blue. The key observation is that $F$ does not have an alternating cycle; indeed, if there were an alternating cycle in $F$, then by removing from $G'$ all red edges of this cycle, and adding to $G'$ all blue edges of this cycle, we would get a graph $G''$ with the same degree sequence as $G'$ and with 
		$|E(G'') \triangle E(G)| < |E(G') \triangle E(G)|$, in contradiction to the minimality of $G'$. Next, observe that for every $i \in [n]$, we have 
		$d_i = d_{G'}(i) = d_G(i) + d^R_F(i) - d^B_F(i)$. \nolinebreak Hence,
		$$
		\sum_{i \in V(F)}{|d^R_F(i) - d^B_F(i)|} = 
		\sum_{i \in V(G)}{|d_G(i) - d_i|} \leq \delta n^2.
		$$
		By Lemma \ref{lem:balanced} we have $|E(G') \triangle E(G)| = e(F) = O(\delta^{1/2}n^2)$, as required. 
	\end{proof}

	\section{Testing Degree Sequence Properties}\label{sec:test}
	In this section we show that every degree-sequence property admits a tester with query complexity independent of $n$, thus proving Theorem \ref{thm:main}.
	Our tester works by approximating the degree sequence of the given input graph. 
	Let us introduce the relevant definitions. 
	A {\em degree statistic} is given by reals $\alpha_1,\dots,\alpha_k \in [0,1]$ with $\alpha_1 + \dots + \alpha_k = 1$. For a degree statistic $\alpha = (\alpha_1,\dots,\alpha_k)$ and an integer $n \geq 1$, we define $d = d(n,\alpha)$ as the sequence $d = (d_1,\dots,d_n) \in \mathbb{N}^n$ which has, for every $1 \leq \ell \leq k$, precisely $\alpha_{\ell} n$ coordinates equal to $\frac{2\ell-1}{2k}n$ (we assume, for simplicity of presentation, that $\frac{n}{2k}$ and $\alpha_1 n,\dots,\alpha_k n$ are integers). For example, the sequence $d = d(n,\alpha)$ corresponding to $\alpha = (0.2,0.8)$ has $0.2n$ terms equal to $\frac{n}{4}$ and $0.8n$ terms equal to $\frac{3n}{4}$. Let $\delta > 0$ and let $G$ be a graph on $[n]$. A degree statistic $\alpha = (\alpha_1,\dots,\alpha_k)$ is said to {\em $\delta$-approximate} $G$ if there is a permutation $\pi : [n] \rightarrow [n]$ such that $\sum_{i=1}^{n}{|d_G(\pi(i)) - d_i|} \leq \delta n^2$, where $(d_1,\dots,d_n) = d(n,\alpha)$. In other words, a degree statistic $\delta$-approximates $G$ if,
 	after possibly relabeling the vertices of $G$, 
 	the degree sequence of $G$ and the degree sequence corresponding to the statistic differ by at most $\delta$ in (normalized) $\ell_1$-distance. 
	
	\begin{lemma}\label{lem:degree statistic}
		There is an algorithm which, given $\delta > 0$ and an input graph $G$, outputs a degree statistic $(\alpha_1,\dots,\alpha_k)$ with $k = \lceil 1/\delta \rceil$ which $\delta$-approximates $G$ with probability at least $\frac{2}{3}$. The query complexity and run-time of the algorithm is $\tilde{O}(\delta^{-8})$.
	\end{lemma}
	\begin{proof}
		We use a standard ``approximation-by-sampling" argument. 
		Put $n := |V(G)|$.
		We set $k := \lceil 1/\delta \rceil$ and $\gamma := \frac{1}{2k(2k+1)}$. 
		Our algorithm works as follows: sample $s := \frac{\log(12k)}{2\gamma^2} = \tilde{O}(\delta^{-4})$ vertices $v_1,\dots,v_s \in V(G)$, and for each $1 \leq i \leq s$, sample $t := \frac{\log(6s)}{2\gamma^2} = \tilde{O}(\delta^{-4})$ additional vertices $u_{i,j} \in V(G)$, $1 \leq j \leq t$, where all samples are made uniformly and independently. 
		Now, query all pairs $\{v_i,u_{i,j}\}$ ($1 \leq i \leq s$ and $1 \leq j \leq t$). 
		For each $1 \leq i \leq s$, let $\bar{d}_i := |N_G(v_i) \cap \{u_{i,1},\dots,u_{i,t}\}|$. 
		Set 
		$\alpha_1 :=
		\frac{1}{s} \cdot 
		\#\{1 \leq i \leq s \; : \; \frac{\bar{d}_i}{t} \leq \frac{1}{k} \}$ and 
		$\alpha_{\ell} := 
		\frac{1}{s} \cdot 
		\#\{1 \leq i \leq s \; : \; \frac{\ell-1}{k} < \frac{\bar{d}_i}{t} \leq \frac{\ell}{k} \}$ 
		for $2 \leq \ell \leq k$. 
		The algorithm outputs $(\alpha_1,\dots,\alpha_k)$.  
		Note that the total number of queries, as well as the run-time, is 
		$O(st) = \tilde{O}(\delta^{-8})$, as required.  
		
		Let us prove the correctness of the above algorithm. 
		Observe that $\bar{d}_i$ is distributed as a binomial random variable with parameters $t$ and $d_G(v_i)/n$. By Hoeffding's inequality, for every $1 \leq i \leq s$ we \nolinebreak have
		\begin{equation}\label{eq:single_vertex_degree_concentration}
		\mathbb{P}\left[ \left|\frac{\bar{d}_i}{t} - \frac{d_G(v_i)}{n}\right| \geq \gamma \right] \leq 
		e^{-2\gamma^2 t} 
		\leq 
		\frac{1}{6s}. 
		\end{equation}
		Let $\mathcal{A}$ be the event that 
		$\left|\frac{\bar{d}_i}{t} - \frac{d_G(v_i)}{n}\right| \leq \gamma$ for every $1 \leq i \leq s$. By using \eqref{eq:single_vertex_degree_concentration} and taking the union bound over $1 \leq i \leq s$, we get that $\mathbb{P}[\mathcal{A}] \geq 5/6$.  
		
		Moving forward, let us introduce some additional notation. For a set $X \subseteq \{0,\dots,n-1\}$, we define
		$c(X) := \frac{1}{n} \cdot \#\{ v \in V(G) : d_G(v) \in X\}$ and 
		$\bar{c}(X) := \frac{1}{s} \cdot \#\{1 \leq i \leq s : d_G(v_i) \in X\}$. Observe that for any given set $X$, the random variable 
		$\#\{1 \leq i \leq s : d_G(v_i) \in X\} = s \cdot \bar{c}(X)$ has a binomial distribution with parameters $s$ and $c(X)$. Hence, by Hoeffding's inequality, we have
		\begin{equation}\label{eq:degree_statistic_concentration}
		\mathbb{P}\left[ \left|\bar{c}(X) - c(X)\right| \geq \gamma \right] \leq 
		e^{-2\gamma^2 s} 
		\leq 
		\frac{1}{12k}. 
		\end{equation}
		We will apply \eqref{eq:degree_statistic_concentration} to the following $2k-2$ sets: $X_{\ell} := \{d : 0 \leq d \leq (\frac{\ell}{k} + \gamma)n\}$ for $1 \leq \ell \leq k-1$ and 
		$Y_{\ell} := \{d : 0 \leq d \leq (\frac{\ell-1}{k} - \gamma)n\}$ for $2 \leq \ell \leq k$. 
		Let $\mathcal{B}$ be the event that $\left|\bar{c}(X) - c(X)\right| \leq \gamma$ for every $X \in \{X_1,\dots,X_{k-1},Y_2,\dots,Y_{k}\}$. By using \eqref{eq:degree_statistic_concentration} and the union bound, we get that $\mathbb{P}[\mathcal{B}] \geq 5/6$. 

		Thus far we have shown that $\mathbb{P}[\mathcal{A} \cap \mathcal{B}] \geq 2/3$. Let us assume from now on that both $\mathcal{A}$ and $\mathcal{B}$ occurred, and show that under this assumption, $(\alpha_1,\dots,\alpha_k)$ indeed $\delta$-approximates $G$.  
		First, observe that for every $1 \leq \ell \leq k-1$, we have
		\begin{equation}\label{eq:X_{ell} estimation}
		\begin{split}
		\frac{1}{n} \cdot \#\left\{v \in V(G) : \frac{d_G(v)}{n} \leq  \frac{\ell}{k} + \gamma \right\} 
		&= 
		c(X_{\ell}) 
		\geq 
		\bar{c}(X_{\ell}) - \gamma 
		\\ &= 
		\frac{1}{s} \cdot \#\left\{1 \leq i \leq s : \frac{d_G(v_i)}{n} \leq  \frac{\ell}{k} + \gamma\right\} - \gamma 
		\\ &\geq 
		\frac{1}{s} \cdot \#\left\{1 \leq i \leq s : \frac{\bar{d}_i}{t} \leq  \frac{\ell}{k}\right\} - \gamma
		\\ &=
		\alpha_1 + \dots + \alpha_{\ell} - \gamma, 
		\end{split}
		\end{equation}
		where the first inequality holds because $\mathcal{B}$ occurred, and the second inequality holds because having $\frac{\bar{d}_i}{t} \leq \frac{\ell}{k}$ implies that $\frac{d_G(v_i)}{n} \leq  \frac{\ell}{k} + \gamma$, as $\mathcal{A}$ occurred. Note that \eqref{eq:X_{ell} estimation} also trivially holds for $\ell = k$, since the left-hand side equals $1$.
		The last equality follows from the definition of $\alpha_1,\dots,\alpha_k$. 
		
		Similarly, for every $2 \leq \ell \leq k$ we have 
		\begin{equation}\label{eq:Y_{ell} estimation}
		\begin{split}
		\frac{1}{n} \cdot \#\left\{v \in V(G) : \frac{d_G(v)}{n} \leq  \frac{\ell-1}{k} - \gamma \right\} 
		&= 
		c(Y_{\ell}) 
		\leq 
		\bar{c}(Y_{\ell}) + \gamma 
		\\ &= 
		\frac{1}{s} \cdot \#\left\{1 \leq i \leq s : \frac{d_G(v_i)}{n} \leq  \frac{\ell-1}{k} - \gamma\right\} + \gamma 
		\\ &\leq 
		\frac{1}{s} \cdot \#\left\{1 \leq i \leq s : \frac{\bar{d}_i}{t} \leq  \frac{\ell-1}{k}\right\} + \gamma 
		\\ &=
		\alpha_1 + \dots + \alpha_{\ell-1} + \gamma. 
		\end{split}
		\end{equation}
		Again, \eqref{eq:Y_{ell} estimation} trivially also holds for $\ell = 1$. 

	We now define subsets $A_1,\dots,A_k \subseteq V(G)$ as follows. Suppose without loss of generality that $V(G) = [n]$ and $d_G(1) \leq \dots \leq d_G(n)$ (otherwise simply relabel the vertices of $G$). For each $1 \leq \ell \leq k$, define 
	$A_{\ell} = \big\{x \in [n] : 
	(\alpha_1 + \dots + \alpha_{\ell-1} + \gamma)n + 1 \leq x \leq
	(\alpha_1 + \dots + \alpha_{\ell} - \gamma)n \big\}$ 
	(some of these sets might be empty). Then 
	$|A_{\ell}| = \max\left\{ 0, (\alpha_{\ell} - 2\gamma)n \right\}$ for every $1 \leq \ell \leq k$, and $A_1,\dots,A_k$ are pairwise disjoint.
	By \eqref{eq:X_{ell} estimation}, there are at least $(\alpha_1 + \dots + \alpha_{\ell} - \gamma)n$ vertices $v \in V(G)$ with $d_G(v) \leq (\frac{\ell}{k} + \gamma) \cdot n$, which implies that $d_G(v) \leq (\frac{\ell}{k} + \gamma) \cdot n$ for every $v \in A_{\ell}$ (as the vertices of $G$ are ordered in an increasing order by their degrees). 
	Similarly, by \eqref{eq:Y_{ell} estimation}, there are at most $(\alpha_1 + \dots + \alpha_{\ell-1} + \gamma)n$ vertices $v \in V(G)$ with $d_G(v) \leq (\frac{\ell-1}{k} - \gamma) \cdot n$, which implies that $d_G(v) > (\frac{\ell-1}{k} - \gamma) \cdot n$ for every $v \in A_{\ell}$. Observe also that 
	$
	\sum_{\ell = 1}^{k}{|A_i|} \geq \sum_{\ell = 1}^{k}{(\alpha_{\ell} - 2\gamma)n} = n - 2k\gamma n.
	$
	Hence, $Z := V(G) \setminus (A_1 \cup \dots \cup A_k)$ has size at most $2k\gamma n$. 
	
	To conclude, we compare the degree sequence of $G$ with the sequence $d(n,\alpha)$, with the goal of showing that $\alpha = (\alpha_1,\dots,\alpha_k)$ indeed $\delta$-approximates $G$. Recall that for each $1 \leq \ell \leq k$, the sequence $d(n,\alpha)$ has exactly $\alpha_{\ell}n$ coordinates equal to $\frac{2\ell-1}{2k}n$. By comparing $|A_{\ell}|$ of these coordinates to the vertices of $A_{\ell}$ (for each $1 \leq \ell \leq k$), and comparing the remaining coordinates to the vertices of $Z$, we see that the $\ell_1$-distance between $d(n,\alpha)$ and (a suitable permutation of) the degree sequence of $G$ is at most
	\begin{align*}
	\sum_{\ell = 1}^{k}
	\sum_{v \in A_{\ell}}{\left| d_G(v) - \frac{2\ell-1}{2k}n\right|} 
	+ |Z|n 
	&\leq 
	\sum_{\ell = 1}^{k}{|A_{\ell}| \cdot \left( \frac{1}{2k} + \gamma \right)n} + |Z|n 
	\leq 
	\left( \frac{1}{2k} + \gamma \right)n^2 + |Z|n 
	\\ &\leq 
	\left( \frac{1}{2k} + (2k+1)\gamma \right)n^2 \leq \delta n^2.
	\end{align*}
	Here, the first inequality follows from the fact that 
	$(\frac{\ell-1}{k} - \gamma) \cdot n < d_G(v) \leq 
	(\frac{\ell}{k} + \gamma) \cdot n$ for every $v \in A_{\ell}$ and 
	$1 \leq \ell \leq k$, and the last inequality follows from our choice of $k$ and $\gamma$. This completes the \nolinebreak proof. 
	\end{proof}

	Recall that a degree-sequence property is a set of realizable sequences which is closed under permuting the coordinates, and that the graph property $\mathcal{P}(\mathcal{D})$ corresponding to a degree-sequence property $\mathcal{D}$ is the set of all graphs whose degree sequence is in $\mathcal{D}$. 
	Recall that the normalized $\ell_1$-distance between sequences $x = (x_1,\dots,x_n)$ and $y = (y_1,\dots,y_n)$ is defined as $\ell_1(x,y) := \frac{1}{n^2} \cdot \sum_{i = 1}^{n}|x_i - y_i|$. 
	The following is the precise form of Theorem \ref{thm:main} that we will prove.
	\begin{theorem}\label{thm:testing}
		Let $\mathcal{D}$ be a degree-sequence property. Suppose that there is an algorithm $\mathcal{A}$ which, given an input consisting of $\delta > 0$, an integer $n$ and a degree statistic $\alpha = (\alpha_1,\dots,\alpha_k)$, returns ``Yes" if there is an $n$-term sequence $d^* \in \mathcal{D}$ such that 
		$\ell_1( d(n,\alpha), d^*) \leq \delta$, returns ``No" if 
		$\ell_1( d(n,\alpha), d^* ) > 2\delta$ for every $n$-term sequence $d^* \in \mathcal{D}$, and runs in time $f(n,\delta,k)$. Then $\mathcal{P}(\mathcal{D})$ is testable with query complexity $q(\varepsilon) = \tilde{O}(\varepsilon^{-16})$ and running time 
		$f\big( n,\Omega(\varepsilon^2),O(\varepsilon^{-2}) \big) + \tilde{O}(\varepsilon^{-16})$. 
	\end{theorem} 
	\begin{proof}
		Our tester for $\mathcal{P}(\mathcal{D})$ works as follows. 
		Let the input to the tester consist of a proximity parameter $\varepsilon > 0$ and an $n$-vertex graph $G$. Let $c$ be the implied constant in the $O$-notation in Theorem \ref{thm:distance}, and set $\delta := \frac{1}{3}(\varepsilon/c)^2 = \Omega(\varepsilon^2)$. 
		We assume that $n \geq \delta^{-2}$; otherwise the tester can just query the entire graph.
		The tester begins by applying the algorithm given by Lemma \ref{lem:degree statistic} and thus obtaining a degree statistic $\alpha = (\alpha_1,\dots,\alpha_k)$ with $k = \lceil 1/\delta \rceil = O(\varepsilon^{-2})$ which $\delta$-approximates $G$ with probability at least $\frac{2}{3}$. This step costs $\tilde{O}(\varepsilon^{-16})$ in both number of queries and run-time (as guaranteed by Lemma \ref{lem:degree statistic}). Next, the tester invokes the algorithm $\mathcal{A}$ with input $\delta$, $n$ and $(\alpha_1,\dots,\alpha_k)$, and accepts if and only if the output of $\mathcal{A}$ is ``Yes".
		
		Let us show that the above is a valid tester for $\mathcal{P}(\mathcal{D})$. To this end, we will prove that if $(\alpha_1,\dots,\alpha_k)$ \linebreak $\delta$-approximates $G$ (which happens with probability at least $2/3$), then the tester works correctly. Suppose first that $G \in \mathcal{P}(\mathcal{D})$, and let us denote by $d^* \in \mathbb{N}^n$ the degree sequence of $G$, noting that $d^* \in \mathcal{D}$. Since $\alpha = (\alpha_1,\dots,\alpha_k)$ $\delta$-approximates $G$, after possibly permuting the vertices of $G$ we have 
		$\ell_1( d(n,\alpha), d^*) \leq \delta$. Therefore, the output of $\mathcal{A}$ must be ``Yes", and the tester hence accepts, \nolinebreak as \nolinebreak required. 
		
		Suppose now that $G$ is $\varepsilon$-far from $\mathcal{P}(\mathcal{D})$. We claim that in this case, we must have $\ell_1( d(n,\alpha), d^* ) > 2\delta$ for every $n$-term sequence $d^* \in \mathcal{D}$. This would imply that $\mathcal{A}$ outputs ``No", and hence the tester rejects, as required. So let us fix any $n$-term sequence $d^* = (d^*_1,\dots,d^*_n) \in \mathcal{D}$, and suppose by contradiction that $\ell_1( d(n,\alpha), d^* ) \leq 2\delta$. Let $d^G = (d_G(1),\dots,d_G(n))$ denote the degree sequence of $G$ (we again assume that $V(G) = [n]$). Since $\alpha = (\alpha_1,\dots,\alpha_k)$ $\delta$-approximates $G$, after possibly permuting the vertices of $G$ we have $\ell_1(d^G,d(n,\alpha)) \leq \delta$. Now, by the triangle inequality, we have 
		$$
		\frac{1}{n^2} \cdot \sum_{i = 1}^{n}{|d_G(i) - d^*_i|} = 
		\ell_1(d^G,d^*) \leq 
		\ell_1(d^G,d(n,\alpha)) + \ell_1(d(n,\alpha), d^*) \leq 3\delta.
		$$ 
		By Theorem \ref{thm:distance}, there is a graph $G^*$ with degree sequence $d^*$ such that 
		$|E(G^*) \triangle E(G)| \leq c \cdot (3\delta)^{1/2} = \varepsilon$ (here we used our choice of $c$ and $\delta$). But $G^* \in \mathcal{P}(\mathcal{D})$ (as $d^* \in \mathcal{D}$), which implies that $G$ is $\varepsilon$-close to $\mathcal{P}(\mathcal{D})$, a contradiction. This completes the proof. 
	\end{proof}

	Note that if $\mathcal{D}$ contains only a single $n$-term sequence $d^*$, then the algorithm $\mathcal{A}$ simply needs to compute the $\ell_1$-distance between $d(n,\alpha)$ and $d^*$. If $d^*$ has only a constant number of distinct terms (and, in particular, if $d^* = (r,\dots,r)$ for some $r$), then this computation can be done in time independent \nolinebreak of \nolinebreak $n$.

\end{document}